\documentclass[11pt,reqno]{amsart}
\usepackage{amssymb,amsmath,amscd}
\usepackage{latexsym,bm,bbm,mathrsfs}
\usepackage{hyperref,graphicx, enumerate}
\usepackage{mathtools}
\usepackage{stmaryrd}
\usepackage{float}
\usepackage[all,pdf]{xy}
\usepackage{graphicx}
\usepackage{tikz-cd}
\usepackage{color}
\graphicspath{ {./} }

\usepackage{stackrel}
\usepackage[left=30mm, right=30mm, top=30mm, bottom=30mm]{geometry}
\usepackage[shortlabels]{enumitem}
\usepackage{ulem}
\usepackage{ dsfont }
\usepackage{verbatim}

\allowdisplaybreaks

\newtheorem{thm}{Theorem}[section]
\newtheorem*{theorem*}{Theorem}

\newtheorem*{quest*}{Question}

\newtheorem{theorem}[thm]{Theorem}
\newtheorem{cor}[thm]{Corollary}
\newtheorem{prop}[thm]{Proposition}
\newtheorem{lemma}[thm]{Lemma}

\newtheorem*{conj*}{Conjecture}

\newtheorem{remark}[thm]{Remark}
\newtheorem{defn}[thm]{Definition}

\newcommand{\GL}{\operatorname{GL}_{2}}
\newcommand{\calB}{\mathcal{B}}

\newcommand{\bba}{{\mathbb{A}}}
\newcommand{\bbc}{{\mathbb{C}}}
\newcommand{\bbf}{{\mathbb{F}}}

\newcommand{\bbq}{{\mathbb{Q}}}

\newcommand{\bbz}{{\mathbb{Z}}}

\newcommand{\Gal}{\operatorname{Gal}}

\newcommand{\Fp}{\bbf_{p}}
\newcommand{\Fpx}{\left(\bbz/p\bbz\right)^{\times}}
\renewcommand{\ss}{\operatorname{ss}}
\newcommand{\Jss}{\operatorname{J}_{\ss}}
\newcommand{\alg}{{\operatorname{alg}}}
\newcommand{\Kalg}{K^{\alg}}
\newcommand{\Kst}{K\left(s,t\right)}
\newcommand{\wt}{\operatorname{wt}}
\newcommand{\tx}{\widetilde{x}}

\newcommand{\fkA}{\mathfrak{A}}
\newcommand{\fkQ}{\mathfrak{Q}}

\usepackage{enumitem}

\title[The automorphism group of the $p^{n}$-torsion points]{The automorphism group of the $p^{n}$-torsion points of an~elliptic curve over a field of characteristic $p \ge 5$}

\author{Bo-Hae Im}
\address{Department of Mathematical Sciences, KAIST, 291 Daehak-ro, Yuseong-gu, Daejeon, 34141, South Korea}
\email{bhim@kaist.ac.kr}
\thanks{Bo-Hae Im was supported by the National Research Foundation of Korea(NRF) grant funded by the Korea government(MSIT) (No.~2020R1A2B5B01001835). Hansol Kim was supported by NSTC grant funded by Taiwan government (MST) (No.~113-2811-M-001-068).}

\author{Hansol Kim}
\address{Institute of Mathematics, Academia Sinica, 6F, Astronomy-Mathematics Building,
No. 1, Sec. 4, Roosevelt Road, Taipei 10617, Taiwan}
\email{jawlang@gate.sinica.edu.tw}

\date{\today}
\subjclass[2020]{Primary 11G05, Secondary 14H52}
\keywords{elliptic curve, torsion subgroup}

\begin{document}

\maketitle

\begin{abstract}
	For a field $K$ of characteristic $p\ge5$ and the elliptic curve $E_{s,t}: y^{2} = x^{3} + sx + t$ defined over the function field $\Kst$ of two variables $s$ and $t$, we prove that for a positive integer $n$, the automorphism group  of the normal extension $\Kst\left(E_{s,t}\left[p^{n}\right]\right)/\Kst$ is isomorphic to  $\left(\bbz/p^{n}\bbz\right)^{\times}$, and its  inseparable degree is $p^{n}$.
\end{abstract}

\section{Introduction}\label{intro}

Elliptic curves play a central role in modern number theory and algebraic geometry, with applications ranging from cryptography to the study of Diophantine equations. An elliptic curve $E$ over a field 
$F$ is a smooth projective variety of genus $1$, equipped with a distinguished point as the identity for an abelian group law defined on its rational points. 

Over fields $F$ of characteristic 0, such as 
$\bbq$ or number fields, elliptic curves $E/F$ have been extensively studied due to their rich arithmetic properties, and significant attention has been given to the faithful mod-$N$ Galois representation 
$$\rho_{E/F,N}: \Gal\left(F\left(E\left[N\right]\right)/F\right) \hookrightarrow \GL\left(\bbz/N\bbz\right)$$ for a positive integer 
$N$.

Serre's open image theorem (\cite{Serre72}) proves that for each non-CM elliptic curve $E$ over $\bbq$, there exists a positive constant integer $M_{E}$, depending only on $E$, such that the image of $\rho_{E/\bbq,N}$ is $\GL\left(\bbz/N\bbz\right)$ for all $N$ with $\gcd\left(N,M_{E}\right) = 1$. Moreover, Serre conjectured the constants $M_{E}$ can be uniformly bounded. This conjecture is known as Serre's uniformity conjecture.
The possible representations of $\rho_{E/\bbq,\ell}$ for primes $\ell$ and elliptic curves defined over $\bbq$ have been  widely studied. For example, it is proven in \cite[Theorem~1]{Mazur} that $\rho_{E/\bbq, \ell}$ is irreducible for each prime $\ell > 163$ and any elliptic curve $E$ defined over $\bbq$. In \cite{GN}, the possible representations $\rho_{E/\bbq,\ell}$ for non-CM elliptic curves $E$ defined over $\bbq$ are classified for each prime $\ell$. Furthermore, in the same work, the possible representations $\rho_{E/\bbq,\ell}$ for CM elliptic curves $E$  over $\bbq$ with endomorphism ring $\mathcal{O}$ (over $\bbc$) are listed for each prime $\ell$. 

Conversely, for each prime $\ell$  and certain subgroups $H$ of  $\GL\left(\bbz/\ell\bbz\right)$, \cite{Z08}  provides equivalent conditions for the $j$-invariants of non-CM elliptic curves $E$ over $\bbq$ such that the image of $\rho_{E/\bbq,\ell}$ corresponds to $H$ (up to conjugate). 
In \cite{LR}, the Galois representations $\rho_{E/\bbq\left(j\right),\ell^{n}}$ is studied for each $j$, an CM elliptic curve $E$ over $\bbq\left(j\right)$, and a prime power $\ell^{n}$.

On the other hand, we consider the elliptic curve $E_{s,t}: y^{2} = x^{3} +sx +t$ defined over $F\left(s,t\right)$, where $s$ and $t$ are independent variables.  Then, it is known  (\cite[\S~7]{DS} or \cite[proof of Lemma~4.1]{IK22}) that 
\begin{align}\label{inc}
	\Gal\left(F\left(s,t\right)\left(E_{s,t}\left[N\right]\right)/F\left(s,t\right)\right) \supseteq \operatorname{SL}_{2}\left(\bbz/N\bbz\right)
\end{align} for each integer $N\ge 2$.

If the characteristic of $F$ is positive, say $\ell$, and if $\ell \nmid N$, then the inclusion~\eqref{inc} also holds by~\cite[6.8.~Corollary~1]{CF}: According to \cite[6.8.~Corollary~1]{CF}, the toroidal compactification $\overline{A_{1,N}}$  of the moduli space of the pairs $\left(E,\calB\right)$ of an elliptic curve and a basis $\calB$ of $E\left[N\right]$ is considered with the given  Weil paring defined over the ring $\bbz\left[\zeta_{N},1/N\right]$. Any prime $\ell$ not dividing~$N$ is a prime in $\bbz\left[\zeta_{N},1/N\right]$. The fiber of $\ell$ is the moduli space of the pairs $\left(E,\calB\right)$ of elliptic curve over a field of characteristic $p$ and a basis $\calB$ of $E\left[N\right]$ with given Weil paring. Since this fiber is irreducible, for the covering $X\left(N\right) \to X\left(1\right)$ of modular curves defined by $\left(E,\calB\right) \mapsto E$, its Galois group $G$ acts transitively on the set of bases of $ E\left[N\right]$ with a given Weil paring $z$, where $z$ is a primitive $N$th root of unity.
The transitivity of this action implies that $G$ 
 contains $\operatorname{SL}_{2}\left(\bbz/N\bbz\right)$.

\

While significant progress has been made in understanding mod-$N$ Galois representations in characteristic $0$, the case of positive characteristic remains less explored in the literature. Most studies center on elliptic curves over characteristic $0$ fields, which strongly motivates us to tackle the more challenging case of positive characteristic.

\

In this paper, we address a more challenging problem in the study of  mod-$N$ Galois representations, particularly when $N$ is a power of the characteristic of the base field. More explicitly, we inverstigate the automorphism group of the normal field extension $\Kst\left(E_{s,t}\left[p^{n}\right]\right)$ fixing $\Kst$, where the base field $K$ has positive characteristic $p \ge 5$, $n$ is a positive integer, and $E_{s,t}$ is the elliptic curve $y^{2} = x^{3} + sx + t$ defined over the function field $\Kst$ of two free variables $s$ and $t$. Our main result is stated as follows:

\begin{theorem}\label{thm:main}
	Let $K$ be a field of characteristic $p \ge 5$. For the elliptic curve $E_{s,t}: y^{2} = x^{3} + sx + t$ defined over $\Kst$, the automorphism group of $\Kst\left(E_{s,t}\left[p^{n}\right]\right)$ fixing $\Kst$ is isomorphic to $\left(\bbz/p^{n}\bbz\right)^{\times}$.
\end{theorem}

Theorem~\ref{thm:main} should be understood in the context of the group structures of $p^{n}$-torsion parts of elliptic curves over a field of a positive characteristic.

Let $E$ be an elliptic curve defined over a field of characteristic $p\ge 2$. It is known that the $p$-torsion subgroup $E\left[p\right]$ is trivial or isomorphic to $\bbz/p\bbz$. We say that $E$ is supersingular if $E\left[p\right]$ is trivial, and ordinary if $E\left[p\right]$ is not trivial. Moreover, for an ordinary $E$ and a positive integer $n$, we have $E\left[p^{n}\right] \cong \bbz/p^{n}\bbz$ (\cite[Ch.III, Corollary~6.4~(c)]{Silverman}). For  ordinary $E$, we denote the automorphism group of the normal extension $K\left(E\left[p^{n}\right]\right)$ over $K$ (i.e., fixing $K$) by $\fkA_{E/K}\left(n\right)$. Since an element of $\fkA_{E/K}\left(n\right)$ is determined by the image of a generator of $E\left[p^{n}\right]  \cong \bbz/p^{n}\bbz$, it follows that $\fkA_{E/K}\left(n\right) $ is isomorphic to a subgroup of $\left(\bbz/p^{n}\bbz\right)^{\times}$. More precisely, for a generator $P_{n}$ of $E\left[p^{n}\right]$, the injective group homomorphism \begin{equation}\label{eqn:rho}
    \rho: \fkA_{E/K}\left(n\right) \hookrightarrow \left(\bbz/p^{n}\bbz\right)^{\times}
\end{equation} satisfies $P_{n}^{\sigma} = \left[\rho\left(\sigma\right)\right]P_{n}$ for all $\sigma \in \fkA_{E/K}\left(n\right)$.

We can derive the following corollaries for elliptic curves defined over fields with positive characteristic:

\begin{cor}\label{cor:main1}
	Let $K$ be a Hilbertian field of characteristic $p\geq 5$. Then, for each integer $n\ge 1$, there exist infinitely many elliptic curves $E/K$  over $K$ with distinct $j$-invariants such that $\fkA_{E/K}\left(n\right) \cong \left(\bbz/p^{n}\bbz\right)^{\times}$.
\end{cor}

We observe that if $E/K$ is supersingular, then clearly $E(K)[p]$ is trivial. However, as noted in Remark~\ref{jss} and Lemma~\ref{lem:ssj} (Section~\ref{sec:back_groud}), there are only finitely many $j$-invariants in $K^{\text{alg}}$, each corresponding to a supersingular elliptic curve over $K$, making the following result particularly noteworthy:
\begin{cor}\label{cor:main2}
	Let $K$ be a Hilbertian field of characteristic $p\geq 5$. Then, there exist infinitely many ordinary elliptic curves $E/K$  over $K$ with distinct $j$-invariants such that $E\left(K\right)\left[p\right]$, and so $E\left(K\right)\left[p^\infty\right]$ is trivial.
\end{cor}

\begin{remark}
	The concept of a Hilbertian field, Hilbert's irreducibility theorem, and related topics are discussed in \cite[\S~9]{Serre97} and \cite[\S~3]{Serre08}. In essence,  if $K$ is Hilbertian, then Hilbert's irreducibility theorem guarantees  infinitely many  points $\left(s,t\right) = \left(a,b\right)\in K\times K$ such that $\fkA_{E_{a,b}/K}\left(n\right) = \fkA_{E_{s,t}/\Kst}\left(n\right) \cong \left(\bbz/p^{n}\bbz\right)^{\times}$ by specialization. A more detailed explanation will be given in the proof of Corollary~\ref{cor:main1} at the end of Section~\ref{sec:proof}. An example of a Hilbertian field with positive characteristic is a global function field (\cite[Corollary~12.8]{FJ86}).
\end{remark}

Key tools used in the  proof of Theorem~\ref{thm:main} include the  supersingular $j$-invariants which are studied in Section~\ref{subsec:ssj}, and the division polynomials of an elliptic curve which are reviewed  in Section~\ref{subsec:division_poly}. After these reviews, we establish several lemmas that facilitate the proof of Theorem~\ref{thm:main}. Finally, the full proof of Theorem~\ref{thm:main} is presented in Section~\ref{sec:proof}.

\

\noindent {\bf Acknowledgement.} We express our gratitude to Professor Chia-Fu Yu for directing us to reference~\cite{CF}, which establishes the inclusion~\eqref{inc} in the setting of fields with positive characteristic.

\section{Supersingular $j$-invariants and division polynomials }\label{sec:back_groud}
Throughout this section,  we assume that $p\ge3$, unless specified otherwise.

\subsection{Supersingular $j$-invariants}\label{subsec:ssj}
First, we define supersingular $j$-invariants and study the properties of polynomial whose zeros are supersingular $j$-invariants that are neither $0$ nor $1728$. A $j_{\ss} \in \Kalg$ is called a supersingular $j$-invariant if there exists a supersingular elliptic curve with $j$-invariant $j_{\ss}$. We denote 
 $$\Jss:=\{\text{supersingular $j$-invarinats in } \Kalg\}.$$

 We recall the following equivalent condition for the supersingularity of an eliiptic curve in  Legendre form;
\begin{lemma}[{\cite[Ch.V, Theorem~4.1~(b)]{Silverman}}]\label{fact:Legendre}
	For a $\lambda \in \Kalg \setminus \left\{ 0,1\right\}$, the elliptic curve $E_{\lambda}: y^{2}= x\left(x-1\right)\left(x-\lambda\right)$ is supersingular if and only if $\lambda$ satisfies $$
		\sum_{k=0}^{\left(p-1\right)/2}\binom{\left(p-1\right)/2}{k}^{2} \lambda^{k} = 0.
	$$
\end{lemma}
\begin{remark}\label{jss} Noting that the $j$-invariant of $E_{\lambda}$ is a rational function of $\lambda$, Lemma~\ref{fact:Legendre} implies that $\Jss$ is finite, $\Jss \subseteq \Fp^{\alg}$, and the absolute Galois group of $\Fp$ acts on $\Jss$ with the trivial action on $\left\{0,1728\right\}$. Moreover, it is known that $\Jss \subseteq \bbf_{p^{2}}$. See \cite[Ch.V, Theorem~3.1~(a)]{Silverman}.
\end{remark}

By Remark~\ref{jss},  the polynomial defined by \begin{equation}\label{fss}
	f_{\ss}\left(j\right) := {\displaystyle \prod_{\substack{j_{\ss}\in \Jss \\ j_{\ss} \ne 0,1728}}} \left(j-j_{\ss}\right)
\end{equation}  over $\Fp^{\alg}$ is separable, and in fact, defined over $\Fp$. Moreover, it is known that $\deg f_{\ss} = \left\lfloor \frac{p-1}{12} \right\rfloor$ by the following Lemma~\ref{lem:ssj}:

\begin{lemma}[{\cite[Ch.V, Theorem~4.1~(c) and its proof]{Silverman}}]\label{lem:ssj} For a prime $p\ge 5$, the following  holds:
	\begin{enumerate}[\normalfont (a)]
		\item $0 \in \Jss$ if and only if $p \equiv 2 \pmod{3}$.
		\item $1728 \in \Jss$ if and only if $p \equiv 3 \pmod{4}$.
		\item $\# \left(\Jss \setminus \left\{0,1728\right\}\right) = \left\lfloor \frac{p-1}{12} \right\rfloor$.
	\end{enumerate}
\end{lemma}

What we have discussed so far in  this section can be summarized as follows:
\begin{lemma}\label{lem:fss}
	The polynomial $f_{\ss}$ is defined over $\Fp$, is separable, and has degree $\left\lfloor \frac{p-1}{12} \right\rfloor$.
\end{lemma}

\subsection{Division polynomials and the formula for the $x$-coordinates under the multiplication-by-$m$ map}\label{subsec:division_poly}
Now, we consider the division polynomials 
associated with the elliptic curves 
$$\mathcal{E}_{s,t}: y^{2} = x^{3} + sx + t \text{ defined over } \bbq\left(s,t\right),  \text{ and  } $$
$$E_{s,t}: y^{2} = x^{3} + sx + t \text{ defined over }\Kst.$$

Since $\mathcal{E}_{s,t}: y^{2} = x^{3} + sx + t$ is  defined over  $\bbz\left[s,t\right]\subset \bbq\left(s,t\right)$, we can consider its mod-$p$ reduction, which is defined over $\Fp\left[s,t\right]$. Given that the characteristic of $K$ is $p$ and a fixed embedding of $\Fp$ to $K$, we obtain natural morphisms from $\mathcal{E}_{s,t}$ to its mod-$p$ reduction $E_{s,t}$ as well as  from $\mathcal{E}_{s,t}\left[p\right]$ to $E_{s,t}\left[p\right]$. These morphisms, induced by the mod-$p$ reduction, are denoted by \begin{equation}\label{red-map}
	\operatorname{red}_{p}: \mathcal{E}_{s,t} \to E_{s,t}, \text{ and } \operatorname{red}_{p}: \mathcal{E}_{s,t}\left[p\right] \to E_{s,t}\left[p\right].
\end{equation}

 For the definitions and properties of division polynomials, we may refer to \cite[\S3]{IK22} or \cite[Ch.III, Exercise~3.7]{Silverman}.

For our purposes, we will review the essential properties of the $m$th division polynomial $\psi_{m,s,t}$ of the elliptic curves $\mathcal{E}_{s,t}$. For each positive integer $m$, the $m$th division polynomial  $\psi_{m,s,t}$ of $\mathcal{E}_{s,t}$ is separable (over an  algebraic closure of the base field over which the given elliptic curve is defined), and its zeros are the $x$-coordinates of the $m$-torsion points.

For odd $m$, $\psi_{m,s,t}\left(x\right) \in \left(\bbz\left[s,t\right]\right)\left[x\right]$ has degree $\frac{m^{2}-1}{2}$ in $x$, and leading coefficient $m$. Its zeros are the $x$-coordinates of the points in $\mathcal{E}_{s,t}\left[m\right] \setminus \left\{O\right\}$.

For even $m$, $\psi_{m,s,t}\left(x,y\right) \in \left(\bbz\left[s,t\right]\right)\left[x,y\right]$ is divisible by $y$, and $\frac{\psi_{m,s,t}}{y} \in \left(\bbz\left[s,t\right]\right)\left[x\right]$ is a polynomial of  degree $\frac{m^{2}-4}{2}$ in $x$, with  leading coefficient $m$, and it is also separable. Its zeros are the $x$-coordinates of the points in $\mathcal{E}_{s,t}\left[m\right] \setminus \mathcal{E}_{s,t}\left[2\right]$.

Another important property of $\psi_{m,s,t} \in \left(\bbz\left[s,t\right]\right)\left[x,y\right]$ is being weight-homogeneous (\cite[\S3]{IK22}). To be more precise, we assign weights to the variables as follows;  $\wt\left(x\right) = 1$, $\wt\left(s\right) = 2$, and $\wt\left(t\right) = 3$. Consequently, 
\begin{align}\label{homog}
	&\psi_{m,s,t}\text{ is a weight-homogeneous polynomial of weight } \frac{m^{2}-1}{2} \text{ for odd } m, \\\notag
	&\text{ while } \frac{\psi_{m,s,t}}{y} \text{ is a weight-homogeneous polynomial of weight } \frac{m^{2}-4}{2} \text{ for even } m.
\end{align}

Additionally, division polynomials play a crucial role in expressing the formulas for the $x$-coordinates under the multiplication-by-$m$ map, which we will denote as the  $\left[m\right]$-map. Before we derive this formula, we note that the product $\psi_{m-1,s,t}\left(x,y\right)\psi_{m+1,s,t}\left(x,y\right)$ is a polynomial in only $s$, $t$, and $x$. For even $m$, both $\psi_{m-1,s,t}\left(x,y\right)$ and $\psi_{m+1,s,t}\left(x,y\right)$ are polynomials in the variables $s$, $t$, and $x$. For odd $m$, we have; \begin{align*}
	\psi_{m-1,s,t}\left(x,y\right)\psi_{m+1,s,t}\left(x,y\right) & = y^{2} \cdot \frac{\psi_{m-1,s,t}\left(x,y\right)}{y} \cdot \frac{\psi_{m-1,s,t}\left(x,y\right)}{y} \\
	\notag & = \left(x^{3} + sx + t\right) \left(\frac{\psi_{m-1,s,t}}{y}\right)\left(x\right) \left(\frac{\psi_{m-1,s,t}}{y}\right)\left(x\right) \in \left(\bbz\left[s,t\right]\right)\left[x\right].    
\end{align*} We denote this product $\psi_{m-1,s,t}\left(x,y\right)\psi_{m+1,s,t}\left(x,y\right)$ by $\left( \psi_{m-1,s,t} \cdot \psi_{m+1,s,t}\right)\left(x\right)$. For a point $\left(x,y\right) \in \mathcal{E}_{s,t}$, we express the $x$-coordinate of the image of $(x,y)$ under the $\left[m\right]$-map as $$
	x\left(\left[m\right]\left(x,y\right)\right) = x - \frac{\left( \psi_{m-1,s,t} \cdot \psi_{m+1,s,t}\right)}{\psi_{m,s,t}^{2}} \left(x\right).
$$ In other words, if $\tx$ is a variable of weight $1$, a zero of the weight-homogeneous polynomial $-\left( \psi_{m-1,s,t} \cdot \psi_{m+1,s,t}\right)\left(x\right) + x\psi_{m,s,t}^{2}\left(x\right) -\tx\psi_{m,s,t}^{2}\left(x\right)$ in the variable $x$ is the $x$-coordinate of a point $\mathcal{S} \in \mathcal{E}_{s,t}$ such that $x\left(\left[m\right]\mathcal{S}\right) = \tx$.

Since we are considering elliptic curves defined over a field with positive characteristic~$p$, it is necessary to examine the mod-$p$ reduction of division polynomials. Obviously, if $\overline{\psi_{m,s,t}\left(x,y\right)}$ denotes the mod-$p$ reduction  of $\psi_{m,s,t}\left(x,y\right)$, then $\overline{\psi_{m,s,t}\left(x,y\right)} \in \left(\Fp\left[s,t\right]\right)\left[x,y\right] \subseteq \left(K\left[s,t\right]\right)\left[x,y\right]$ and it  has the same general properties as described previously, except for those concerning the  leading coefficients and degrees. In the case when $m=p$, we are interested in the following mod-$p$ polynomials: 
\begin{equation} \label{eqn:poly1} 
	\overline{\psi_{p,s,t}\left(x\right)} \text{ and }
\end{equation}
\begin{equation} \label{eqn:poly2} 
	\overline{-\left(\psi_{p-1,s,t}\cdot\psi_{p+1,s,t}\right)\left(x\right) + x\psi_{p,s,t}^{2}\left(x\right) -\tx\psi_{p,s,t}^{2}\left(x\right)}.
\end{equation}
We note that the zeros of the polynomial \eqref{eqn:poly1} are the $x$-coordinates of the points in~$E_{s,t}$ of order $p$,  and the zeros of the polynomial \eqref{eqn:poly2} are the $x$-coordinates of points $S$ in $E_{s,t}$ such that $x\left(\left[p\right]S\right) = \tx$.

Before finding the explicit descriptions of \eqref{eqn:poly1} and \eqref{eqn:poly2}, we introduce some notation. Note that the group $\mathcal{E}_{s,t}\left[p\right]$ is isomorphic to $\left(\bbz/p\bbz\right)^{2}$. Moreover, since we focus on the $x$-coordinates of the $p$-torsion points of $\mathcal{E}_{s,t}$, and for any point $\mathcal{R} \in \mathcal{E}_{s,t}$ the $x$-coordinates of $\mathcal{R}$ and $-\mathcal{R}$ are the same,  we identify the quotient $\left(\bbz/p\bbz\right)^{2}/\left\{\pm1\right\}$ as the set of cosets of $\mathcal{E}_{s,t}\left[p\right]$ in which $\mathcal{R}$ is identified with $-\mathcal{R}$. The elements of $\left(\bbz/p\bbz\right)^{2}/\left\{\pm1\right\}$ will be denoted by $\overline{\left(a,b\right)}$. Similarly, the elements of $\left(\bbz/p\bbz\right)^{\times}/\left\{\pm1\right\}$ are denoted by $\overline{a}$ where there is no confusion.

We fix an algebraic closure $\Kalg$ of $K$. To clarify the notion of weights and degrees: variables such as $s$, $t$, $\tx$, $x$, and $X$ will appear, and the weight of a monomial in these variables refers to the sum of the  weights assigned when we define them; $$
	\wt\left(s\right) = 2, ~ \wt\left(t\right) = 3, ~\wt\left(\tx\right) = 1, ~\wt\left(x\right) = 1,~ \text{ and } ~\wt\left(X\right) = p.
$$ We say that a polynomial is {\it weight-homogeneous} if it is a sum of monomials of the same weight. All polynomials to be discussed will be weight-homogeneous. If we consider the polynomials as polynomials in a single variable $x$ or $X$ with  coefficients that are polynomials in $s$, $t$, and $\tx$, then the term `degree' refers to the degree in $x$ or $X$.

Now, we give the description of the polynomials~\eqref{eqn:poly1} and~\eqref{eqn:poly2}, starting with that of~\eqref{eqn:poly1} as follows:

\begin{lemma}\label{lem:poly1}
	Let $X$ be a variable of weight $p \ge 3$. Then, there exists a weight-homogeneous polynomial $$
		\theta_{s,t}\left(X\right) := a_{\frac{p-1}{2}} X^{\frac{p-1}{2}} + \left(\sum_{i=1}^{\frac{p-3}{2}}a_{\frac{p-1}{2}+pi} X^{\frac{p-1}{2}-i}\right) + a_{\frac{p^{2}-1}{2}},
	$$ with  weight-homogeneous polynomial coefficients $a_{k} = a_{k}\left(s,t\right) \in \Fp\left[s,t\right]$ of weight $k$ such that $$
		\theta_{s,t}\left(x^{p}\right)=\overline{\psi_{p,s,t}\left(x\right)} \text{ of \eqref{eqn:poly1}}.
	$$
\end{lemma}
\begin{proof}
	For convenience, set  $h:=\frac{p^{2}-1}{2}$. Since $\psi_{p,s,t}\left(x\right)$ is weight-homogeneous of weight $h$ by~\eqref{homog}, there are weight-homogeneous polynomials $A_{k} \in \bbz\left[s,t\right]$ of weight $k$ such that 
 \begin{equation}\label{eq:psi}
		\psi_{p,s,t}\left(x\right) = \sum_{k=0}^{h} A_{k} x^{h-k}.
	\end{equation}
We show that $p \nmid A_{h}$. If $p \mid A_{h}$, then the polynomial $\overline{\psi_{p,s,t}}\left(x\right) \in \left(\Fp\left[s,t\right]\right)\left[x\right]$ would be divisible by $x$. This implies that for any $s_{0},t_{0} \in \Kalg$ such that $4s_{0}^{3}+27t_{0}^{2} \ne 0$, $0$ is a zero of $\overline{\psi_{p,s,t}}\left(x\right)$ and $E_{s_{0},t_{0}}\left(\Kalg\right)$ has a point $\left(0,\sqrt{t_{0}}\right)$ of order~$p$. This contradicts the existence of a supersingular elliptic curve by Lemma~\ref{lem:ssj}. Hence, $p \nmid A_{h}$ and so $A_{h} \ne 0$. So we consider the monic polynomial $$
	\Psi_{s,t}\left(x\right) := \sum_{k=0}^{h} \frac{A_{k}}{A_{h}} x^{k} = \frac{\psi_{p,s,t}\left(x^{-1}\right)}{A_{h}}x^{h}
$$ with coefficients in the local Dedekind domain $$
	D := \left\{a/d: a,d \in \bbz\left[s,t\right], p \nmid d\right\}.
$$
Now, considering the group structures of $E_{s,t}\left[p\right] \cong \bbz/p\bbz$ and $\mathcal{E}_{s,t}\left[p\right] \cong \left(\bbz/p\bbz\right)^{2}$, the kernel of  the map, $\operatorname{red}_{p}: \mathcal{E}_{s,t}\left[p\right] \to E_{s,t}\left[p\right]$ given in \eqref{red-map} contains a non-trivial point $\mathcal{Q}_{1}$. Then, we fix a point $\mathcal{P}_{1} \in \mathcal{E}_{s,t}\left[p\right]$ such that $\mathcal{P}_{1}$ and $\mathcal{Q}_{1}$ generate $\mathcal{E}_{s,t}\left[p\right]$. The zeros of $\Psi_{s,t}$ are the reciprocals of the zeros of $\psi_{p,s,t}$, which are the $x$-coordinates $x\left(\left[a\right]\mathcal{P}_{1}+\left[b\right]\mathcal{Q}_{1}\right)$ of  $p$-torsion points $\left[a\right]\mathcal{P}_{1}+\left[b\right]\mathcal{Q}_{1}$ for $\overline{\left(a,b\right)} \in \left(\left(\bbz/p\bbz\right)^{2}/\left\{\pm 1\right\}\right) \setminus \left\{\overline{\left(0,0\right)}\right\}$. We note that since $A_{h} \ne 0$, any zero of $\psi_{p,s,t}$ is non-zero, and so  \begin{align*}
		\Psi_{s,t}\left(x\right)
		& = \prod_{\overline{\left(a,b\right)} \in \left(\left(\bbz/p\bbz\right)^{2}/\left\{\pm 1\right\}\right) \setminus \left\{\overline{\left(0,0\right)}\right\} } \left(x - \frac{1}{x\left(\left[a\right]\mathcal{P}_{1}+\left[b\right]\mathcal{Q}_{1}\right)}\right)\\
		& = \prod_{\overline{b} \in \Fpx/\left\{\pm 1\right\}} \left(x - \frac{1}{x\left(\left[b\right]\mathcal{Q}_{1}\right)} \right)
		\prod_{\overline{a} \in \Fpx/\left\{\pm 1\right\}} \prod_{b \in \bbz/p\bbz}\left(x -\frac{1}{x\left(\left[a\right]\mathcal{P}_{1}+\left[b\right]\mathcal{Q}_{1}\right)} \right).
	\end{align*} 
Since $\Psi_{s,t}$ is monic and its coefficients lie in $D$, if we let $\mathfrak{D}$ be the integral closure  of $D$ in the splitting field of $\Psi_{s,t}\left(x\right)$ over the field of fractions of $D$, then the zeros of  $\Psi_{s,t}$  belong to  $\mathfrak{D}$. For a prime ideal $\mathfrak{p} \subseteq \mathfrak{D}$ lying above the prime ideal $pD \subseteq D$, and for the image $P_{1}$ of $\mathcal{P}_{1}$ under~$\operatorname{red}_{p}$, we have: $$\frac{1}{x\left(\mathcal{P}_{1}\right)} \equiv \frac{1}{x\left(P_{1}\right)} \pmod{\mathfrak{p}}, ~\text{ and } ~\frac{1}{x\left(\mathcal{Q}_{1}\right)} \equiv 0 \pmod{\mathfrak{p}}.$$ Therefore, we obtain modulo $\mathfrak{p}$; \begin{align*}
		\Psi_{s,t}\left(x\right)
		& \equiv \left(\prod_{\overline{b} \in \Fpx/\left\{\pm 1\right\}} x\right) \prod_{\overline{a} \in \Fpx/\left\{\pm 1\right\}} \prod_{b \in \bbz/p\bbz} \left(x - \frac{1}{x\left(\left[a\right]P_{1} \right)}\right)\\
		& \equiv x^{\frac{p-1}{2}} \prod_{\overline{a} \in \Fpx/\left\{\pm 1\right\}} \left(x^{p} - \frac{1}{\left(x\left(\left[a\right]P_{1}\right)\right)^{p}}\right) \pmod{\mathfrak{p}}.
	\end{align*} Consequently, we have $$
		\psi_{p,s,t}\left(x\right) \equiv A_{h} \cdot \Psi_{s,t}\left(x^{-1}\right) x^{h} \equiv A_{h} \prod_{\overline{a} \in \Fpx/\left\{\pm 1\right\}} \left(1 - \frac{x^{p}}{\left(x\left(\left[a\right]P_{1}\right)\right)^{p}}\right) \pmod{\mathfrak{p}},
	$$
 which implies that $\overline{\psi_{p,s,t}\left(x\right)}$ is a polynomial in $x^{p}$, so from~\eqref{eq:psi}, $\overline{\psi_{p,s,t}\left(x\right)} = {\displaystyle \sum_{i=0}^{\frac{p-1}{2}} } \overline{A_{h - pi}} x^{pi}$. 
By letting $\theta_{s,t}\left(X\right):= a_{\frac{p-1}{2}}X^{\frac{p-1}{2}} + \left( {\displaystyle \sum_{i=0}^{\frac{p-1}{2}}} a_{\frac{p-1}{2}+pi} X^{\frac{p-1}{2} - i}\right) + a_{\frac{p^{2}-1}{2}}$, where $a_{k} = \overline{A_{k}} \in \Fp\left[s,t\right]$, it completes the proof.
\end{proof}

Next, we give the description of the polynomial~\eqref{eqn:poly2}.
\begin{lemma}\label{lem:poly2}
	Let $\tx$ and $X$ be  variables of weight $1$ and of weight $p \ge 3$, respectively. Then, there exists  a weight-homogeneous polynomial $$
		\eta_{s,t,\tx}\left(X\right) := X^{p} + \left(\sum_{i=1}^{p-1}\left(b_{pi} + \tx c_{pi-1}\right) X^{p-i}\right) + \left(b_{p^{2}} + \tx c_{p^{2}-1}\right),
	$$ with weight-homogeneous polynomial coefficients $b_{k},c_{k} \in \Fp\left[s,t\right]$ of weight $k$ such that $$
		\eta_{s,t,\tx}\left(x^{p}\right) =	\overline{-\left(\psi_{p+1,s,t}\psi_{p-1,s,t}\right)\left(x\right) + x\psi_{p,s,t}^{2}\left(x\right) - \tx \psi_{p,s,t}^{2}\left(x\right)} \text{ of \eqref{eqn:poly2}}.
	$$ In particular, the polynomial coefficients $c_{k}$ satisfy  $$
		-\left(\overline{\psi_{p,s,t}(x)}\right)^2=-\left(x^{p\frac{p-1}{2}} + \left(\sum_{i=1}^{\frac{p-3}{2}}a_{\frac{p-1}{2}+pi} x^{p\frac{p-1}{2}-pi}\right) + a_{\frac{p^{2}-1}{2}}\right)^{2} = \left(\sum_{i=1}^{p-1}c_{pi-1} x^{p^{2}-pi}\right) + c_{p^{2}-1},
	$$ where the coefficients $a_{k}$ are given in Lemma~\ref{lem:poly1}.
\end{lemma}
\begin{proof}
	First, note that the polynomial \begin{align*}
		\left(\psi_{p+1,s,t}\cdot \psi_{p-1,s,t}\right)\left(x\right) = \left(x^{3}+sx+t\right) \left(\frac{\psi_{p-1,s,t}}{y}\right)\left(x\right) \left(\frac{\psi_{p+1,s,t}}{y}\right)\left(x\right)
	\end{align*} has degree $p^{2}$ in $x$ and leading coefficient $p^{2}-1$. We choose two points $\mathcal{P}_{1}$ and $\mathcal{Q}_{1}$ which generate $\mathcal{E}_{s,t}\left[p\right]$ as in the proof of Lemma~\ref{lem:poly1}. 
	For a point  $\mathcal{R} \in \mathcal{E}_{s,t}\left(\bbq\left(s,t,\tx\right)^{\alg}\right)$ whose $x$-coordinate is $\tx$, let a point $\mathcal{S} \in \mathcal{E}_{s,t}\left(\bbq\left(s,t,\tx\right)^{\alg}\right)$ be such that $\left[p\right] \mathcal{S} = \mathcal{R}$. Then, the zeros of the polynomial \begin{equation}\label{mid-poly}
		\sigma_{s,t,\tx}\left(x\right) := -\left(\psi_{p+1,s,t} \cdot \psi_{p-1,s,t}\right)\left(x\right) + x\psi_{p,s,t}^{2}\left(x\right) - \tx \psi_{p,s,t}^{2}\left(x\right)
	\end{equation}  are the $x$-coordinates of points $\mathcal{S}+\left[a\right]\mathcal{P}_{1}+\left[b\right]\mathcal{Q}_{1}$ for $\left(a,b\right) \in \left(\bbz/p\bbz\right)^{2}$ as explained between \eqref{homog} and \eqref{eqn:poly1}. By comparing the degrees in $x$ of $-\left(\psi_{p+1,s,t}\psi_{p-1,s,t}\right)\left(x\right)$, $x\psi_{p,s,t}^{2}\left(x\right)$, and $\tx \psi_{p,s,t}^{2}\left(x\right)$, we deduce  that the degree of $\sigma_{s,t,\tx}$ is  $p^{2}$ and its leading coefficient is $1-p^{2}$. Therefore, we have $$
		\frac{1}{1-p^{2}}\sigma_{s,t,\tx}\left(x\right)
		= \prod_{\left(a,b\right) \in \left(\bbz/p\bbz\right)^{2}} \left(x-x\left(\mathcal{S} + \left[a\right]\mathcal{P}_{1}+\left[b\right]\mathcal{Q}_{1}\right)\right).
	$$ Let  $D : = \left\{a/d: a,d \in \bbz\left[s,t,\tx\right], p\nmid d \right\}$ which is a local Dedekind domain with the field of fractions $\bbq\left(s,t,\tx\right)$, and let  $\mathfrak{D}$ be the integral closure of $D$ in the splitting field of $\frac{1}{1-p^{2}}\sigma_{s,t,\tx}\left(x\right)$ over $\bbq\left(s,t,\tx\right)$. Since  $\frac{1}{1-p^{2}}\sigma_{s,t,\tx}\left(x\right)$  is monic and integral over $D$, its zeros belong to  $\mathfrak{D}$. 
 
	For a prime ideal $\mathfrak{p} \subseteq \mathfrak{D}$ lying above the prime ideal $pD $ of $D$, let  $P_{1}$ be the images of $\mathcal{P}_{1}$, and $S$ be the image of $\mathcal{S}$ under $\operatorname{red}_{p}: \mathcal{E}_{s,t} \to E_{s,t}$ given in \eqref{red-map}. Then, we have \begin{align*}
		\overline{\sigma_{s,t,\tx}\left(x\right)}\equiv \frac{1}{1-p^{2}}\sigma_{s,t,\tx}\left(x\right)
		& \equiv \prod_{a \in \bbz/p\bbz}\prod_{b \in \bbz/p\bbz} \left(x-x\left(S + \left[a\right]P_{1}\right)\right)\\
		& \equiv \prod_{a \in \bbz/p\bbz} \left(x^{p}-\left(x\left(S+\left[a\right]P_{1}\right)\right)^{p}\right) \pmod{\mathfrak{p}}.
	\end{align*} 
	Thus, $\overline{\sigma_{s,t,\tx}\left(x\right)}$ is a monic polynomial in $x^{p}$. Moreover, by Lemma~\ref{lem:poly1}, $\overline{\psi_{p,s,t}^{2}\left(x\right)}$ is also a polynomial in $x^{p}$. Hence,   $\overline{-\left(\psi_{p+1,s,t}\cdot \psi_{p-1,s,t}\right)\left(x\right) + x\psi_{p,s,t}^{2}\left(x\right)}$ is a polynomial in $x^p$, since it is equal to $\overline{\sigma_{s,t,\tx}\left(x\right)} + \tx \overline{\psi_{p,s,t}^{2}\left(x\right)}$. In conclusion, we have polynomial coefficients $b_{k},c_{k} \in \Fp\left[s,t\right]$ of weight $k$ such that $$
		\overline{-\left(\psi_{p+1,s,t}\cdot \psi_{p-1,s,t}\right)\left(x\right) + x\psi_{p,s,t}^{2}\left(x\right)} = x^{p^{2}} + b_{p} x^{p(p-1)} + b_{2p} x^{p(p-2)} + \cdots + b_{p^{2}-p} x^{p} + b_{p^{2}}
	$$ and $$
		-\overline{\psi_{p,s,t}\left(x\right)}^{2} = c_{p-1} x^{p(p-1)} + c_{2p-1} x^{p(p-2)} + \cdots + c_{p^{2}-p-1} x^{p} + c_{p^{2}-1}, 
	$$ which proves the last statement, and thus, referring to \eqref{mid-poly}, the polynomial $\overline{\sigma_{s,t,\tx}\left(x\right)}$ takes the form, $$
		\overline{\sigma_{s,t,\tx}\left(x\right)} = x^{p^{2}} + \left(\sum_{i=1}^{p-1}\left(b_{pi} + \tx c_{pi-1}\right) x^{p^{2}-pi}\right) + \left(b_{p^{2}} + \tx c_{p^{2}-1}\right)=\eta_{s,t,\tx}(x^p),
	$$
 which completes the proof.
\end{proof}

\section{Proofs of Theorem~\ref{thm:main}, Corollary~\ref{cor:main1}, and Corollary~\ref{cor:main2}}\label{sec:proof}

Throughout this section,  we assume that the base field $K$ has  characteristic $p\geq 5$. To prove Theorem~\ref{thm:main}, we consider the following tower of fields:
{\small 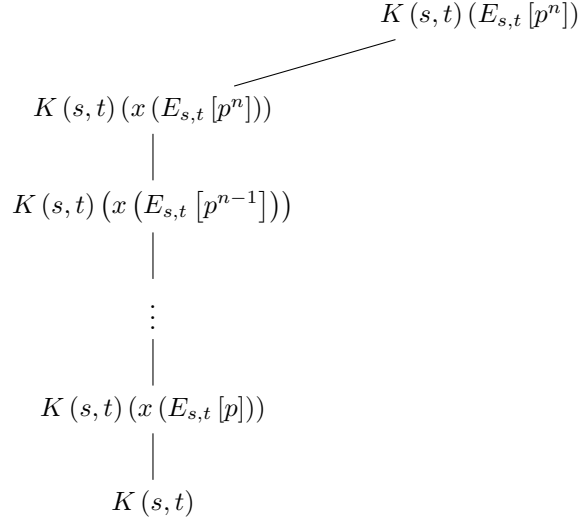
\begin{figure}[H]
	\centering
		\begin{tikzcd}
			& {\Kst\left(E_{s,t}\left[p^{n}\right]\right)} \\
			{\Kst\left(x\left(E_{s,t}\left[p^{n}\right]\right)\right)} \arrow[ru, no head]  &                                              \\
			{\Kst\left(x\left(E_{s,t}\left[p^{n-1}\right]\right)\right)} \arrow[u, no head] &                                              \\
			\vdots \arrow[u, no head]                                                       &                                              \\
			{\Kst\left(x\left(E_{s,t}\left[p\right]\right)\right)} \arrow[u, no head]       &                                              \\
			\Kst \arrow[u, no head]                                                         &                                             
		\end{tikzcd}
	\caption{The tower of subfields of $\Kst\left(E_{s,t}\left[p^{n}\right]\right)$ containing $\Kst$}
	\label{fig:tower}
\end{figure}
}
The tower in Figure~\ref{fig:tower} involves three types of extensions: \begin{itemize}
	\item $\Kst\left(E_{s,t}\left[p\right]\right)/\Kst$,
	\item $\Kst\left(x\left(E_{s,t}\left[p^{k+1}\right]\right)\right)/\Kst\left(x\left(E_{s,t}\left[p^{k}\right]\right)\right)$ for positive $k$, and 
	\item $\Kst\left(E_{s,t}\left[p^{n}\right]\right)/\Kst\left(x\left(E_{s,t}\left[p^{n}\right]\right)\right)$.
\end{itemize}

\

The main focus of our analysis is  the separable and inseparable degrees of these field extensions, as detailed in Corollary~\ref{cor:first_extn}, Corollary~\ref{cor:eta_irr}, and Proposition~\ref{prop:modified_main}, which  lead us to conclude that the separable and inseparable degrees of the extension $\Kst\left(E_{s,t}\left[p^{n}\right]\right)$ over $\Kst$ are $p^{n}$ and $p^{n-1} \left(p-1\right)$, respectively. Since the order of the automorphism group of a normal extension equals its separable degree, and the automorphism group of $\Kst\left(E_{s,t}\left[p^{n}\right]\right)$ over $\Kst$ is  a subgroup of $\left(\bbz/p^{n}\bbz\right)^{\times}$ whose order is $p^{n-1} \left(p-1\right)$, this completes the proof of Theorem~\ref{thm:main}.

\

First, we consider the field extension $\Kst\left(x\left(E_{s,t}\left[p\right]\right)\right)/\Kst$. To do this,  we examine  the relations among 
the  polynomial coefficients~$a_{k}\left(s,t\right)$  of  the polynomial~$\theta_{s,t}$ in Lemma~\ref{lem:poly1}.

\begin{prop}\label{prop:leading_coeff_and_other_coeff} Considering the coefficients $a_{k}\left(s,t\right)$ of the polynomial $\theta_{s,t}\left(X\right)$ obtained in Lemma~\ref{lem:poly1}, 
	if a point  $\left(s_{0},t_{0}\right) \in \left(\Kalg\right)^{2}$ is a zero of $a_{\frac{p-1}{2}}\left(s,t\right)$, then it is also a zero of  $a_{\frac{p-1}{2}+pi}\left(s,t\right)$ for all $1\le i \le \frac{p-3}{2}$, but it is not a zero of $a_{\frac{p^{2}-1}{2}}\left(s,t\right) $, i.e., $a_{\frac{p^{2}-1}{2}}\left(s_{0},t_{0}\right) \ne 0$.
\end{prop}
\begin{proof}
	Since the degree of the polynomial $\theta_{s_{0},t_{0}}\left( X\right)$ specialized  at $(s_0, t_0)$ is less than $\frac{p-1}{2}$, it follows that  $\Big| E_{s_{0},t_{0}}\left[p\right] \Big| < p$. This implies that the specialized elliptic curve $E_{s_{0},t_{0}}$ is supersingular and that $\theta_{s_{0},t_{0}}\left( X\right)$ has no zero in $\Kalg$. In other words, $\theta_{s_{0},t_{0}}\left( X\right)$ is a non-zero constant polynomial. Therefore, $a_{\frac{p-1}{2}+pi}\left(s_{0},t_{0}\right) = 0$ for all $1 \le i \le \frac{p-3}{2}$, and $a_{\frac{p^{2}-1}{2}}\left(s_{0},t_{0}\right) \ne 0$.
\end{proof}

In the proof of Proposition~\ref{prop:leading_coeff_and_other_coeff}, we observe that the coefficient $a_{\frac{p-1}{2}}$ is connected to the supersingular property of elliptic curves. To study this connection, we introduce two values $e_{3}$ and $e_{4}$, which depend on $p$:
\begin{defn}\label{defn:idx}
	We define $$
		e_{3} := \begin{cases}0,&\text{ if } p \equiv 1\pmod{3}, \\ 1,&\text{ if } p \equiv 2\pmod{3},\end{cases} \text{ and } e_{4} := \begin{cases}0,&\text{ if } p \equiv 1\pmod{4}, \\ 1,&\text{ if } p \equiv 3\pmod{4}.\end{cases}
	$$ 
\end{defn} 

Since $a_{\frac{p-1}{2}}$ is weight-homogeneous with weight $\frac{p-1}{2}$,  it follows that $s^{e_{3}}t^{e_{4}}$ divides $a_{\frac{p-1}{2}}$. The quotient $\frac{a_{\frac{p-1}{2}}}{s^{e_{3}}t^{e_{4}}}$ has weight  $6 \left\lfloor \frac{p-1}{12} \right\rfloor$. Therefore, we have a weight-homogeneous polynomial $B\left(U,V\right)$ of weight $\left\lfloor \frac{p-1}{12} \right\rfloor$, where $\wt\left(U\right) = 1$ and $\wt\left(V\right) = 1$ such that \begin{equation}\label{eqn:leading_coeff}
	a_{\frac{p-1}{2}}\left(s,t\right) = s^{e_{3}}t^{e_{4}} B\left(s^{3},t^{2}\right).    
\end{equation}

Next, we determine the form of $B\left(U,V\right)$, noting that $B$ is clearly defined over $\Fp$.

\begin{prop}\label{prop:leading_coeff_and_ssj}
	$B\left(U,V\right) = C F_{\ss}\left(U, \frac{U+27V/4}{1728}\right)$ for some non-zero constant $C\in \Fp$, where $F_{\ss}\left(J,Z\right)$ is the homogenization of $f_{\ss}\left(j\right)$ from \eqref{fss} with respect to $j=J/Z$.
\end{prop}
\begin{proof}
	We let $b\left(u\right)$ be the dehomogenization of $B\left(U,V\right)$ where $u = U/V$. It is sufficient to show that two polynomials $b\left(u\right)$ and $g\left(u\right): = F_{\ss}\left(u, \frac{u+27/4}{1728}\right)$ in $\Fp\left[u\right]$ are equal up to a non-zero scalar multiple. By Lemma~\ref{lem:fss} and the definition \eqref{fss}, we have $\deg b \le \left\lfloor \frac{p-1}{12}\right\rfloor = \deg g$. If $g \mid b$ in $\Fp\left[u\right]$, then we conclude that $b$ is the product of $g$ and a non-zero constant. So we show that $g$ divides $b$. By Lemma~\ref{lem:fss}, $f_{\ss}$ and $g$ are separable, and therefore, if all the zeros of $g$ are also the zeros of $b$, then $g$ divides $b$. Thus, it is enough to show that all zeros of $g$ are the zeros of $b$.


	For any zero $u_{0} \in \Kalg$ of $g$, we have $j_{\ss} := \frac{1728u_{0}}{u_{0}+27/4} \in \Jss \setminus \left\{0,1728\right\}$. This implies that $u_{0} \ne 0$. The $j$-invariant of the specialized elliptic curve $E_{u_{0},u_{0}}$ is precisely $j_{\ss}$, and since this curve is supersingular, the polynomial $\theta_{u_{0},u_{0}}$ must be a non-zero constant. Therefore, we conclude that $B\left(u_{0}^{3},u_{0}^{2}\right) = 0$, so $b\left(u_{0}\right)=0$.
\end{proof}
Proposition~\ref{prop:leading_coeff_and_ssj} leads to the following useful corollary.

\begin{cor}\label{cor:leading_coeff_is_sep}
	Consider the following weight-homogeneous polynomial coefficients  $a_i \in K\left[s,t\right]$ of the polynomial $\theta_{s,t}\left(X\right)$ obtained in Lemma~\ref{lem:poly1}.
 Then, the following holds: \begin{enumerate}[\normalfont (a)]
		\item For any prime divisor $Q$ of $a_{\frac{p-1}{2}}$, we have $Q^{2} \nmid a_{\frac{p-1}{2}}$.
		\item $a_{\frac{p-1}{2}} $ divides $ a_{\frac{p-1}{2}+pi}$, for $1 \le i \le \frac{p-3}{2}$.
		\item The polynomials $a_{\frac{p-1}{2}}$ and $a_{\frac{p^{2}-1}{2}}$ are relatively prime.
	\end{enumerate}
\end{cor}
\begin{proof}
	Before proceeding with the proof, we note two properties of the prime divisor $Q\left(s,t\right)$ of $a_{\frac{p-1}{2}}$.  First, $Q\left(s,t\right)$ either divides $B\left(s^{3},t^{2}\right)$, or is equal to $s$ or $t$, since $a_{\frac{p-1}{2}}\left(s,t\right) =s^{e_{3}}t^{e_{4}}B\left(s^{3},t^{2}\right)$, referring to \eqref{eqn:leading_coeff}. Second,  if $Q\left(s,t\right) \ne s,t$, then $Q\left(s,t\right) = Q_{0}\left(s^{3},t^{2}\right)$ for some weight-homogeneous polynomial $Q_{0}\left(U,V\right)$, where the weights are assigned such that $\wt\left(U\right) = \wt\left(V\right) = 1$. For the case $Q\left(s,t\right) \ne s,t$, let $q_{0}\left(u\right)$ and $b\left(u\right)$ be the dehomogenizations of $Q_{0}\left(U,V\right)$ and  $B\left(U,V\right)$, respectively, with respect to $u = U/V$.
 
	To prove (a), if $Q\left(s,t\right) = s$, then $s$ divides $B\left(s^{3},t^{2}\right)$, so $s^{3}$ would divide $B\left(s^{3},t^{2}\right)$ and this would imply $$
		U ~~\bigg|~~ B\left(U,V\right) = CF_{\ss}\left(U, \frac{U+27V/4}{1728}\right) = C\prod_{\substack{j_{\ss}\in \Jss \\ j_{\ss} \ne 0,1728}} \left(U-j_{\ss}\frac{U+27V/4}{1728}\right)
	$$ by Proposition~\ref{prop:leading_coeff_and_ssj}, which is impossible. A similar contradiction arises if $Q\left(s,t\right) = t$. Now, we have that $Q\left(s,t\right)$ divides $B\left(s^{3},t^{2}\right)$. Then, this implies that $Q^{2}\left(s,t\right) \mid B\left(s,t\right)$, since $Q\left(s,t\right)=Q_{0}\left(s^{3},t^{2}\right)$, and that $q_{0}^{2}\left(u\right) \mid b\left(u\right)$. However, this is a contradiction, since we know that $b\left(u\right)$ is separable by Proposition~\ref{prop:leading_coeff_and_ssj} and Lemma~\ref{lem:fss}.

	To prove parts (b) and (c),  it is enough to show that any prime divisor $Q\left(s,t\right)$ of $a_{\frac{p-1}{2}}\left(s,t\right)$ divides $a_{\frac{p-1}{2}+pi}\left(s,t\right)$ for each $1\le i \le \frac{p-3}{2}$, but not $a_{\frac{p^{2}-1}{2}}$, by considering (a). If $Q\left(s,t\right) = s$, then $a_{\frac{p-1}{2}}\left(0,t\right) = 0$. By Proposition~\ref{prop:leading_coeff_and_other_coeff}, it follows that $a_{\frac{p-1}{2}+pi}\left(0,t\right) = 0$ while  $a_{\frac{p^{2}-1}{2}}\left(0,t\right) \ne 0$. These occur only when $s \mid a_{\frac{p-1}{2}+pi}\left(s,t\right)$ and $s \nmid a_{\frac{p^{2}-1}{2}}\left(s,t\right)$. Similarly, if $Q\left(s,t\right) = t$, then $t \mid a_{\frac{p-1}{2}+pi}\left(s,t\right)$ and $t \nmid a_{\frac{p^{2}-1}{2}}\left(s,t\right)$.	For a $Q\left(s,t\right)$ dividing $B\left(s^{3},t^{2}\right)$, we note that $q_{0}\left(u\right)$ is irreducible over $\Kst$. Let $u_{0} \in \Kalg$ be a zero of $q_{0}\left(u\right)$. If $u_{0}=0$, then $q_{0}\left(u\right) = u$, implying that  $U = Q_{0}\left(U,V\right)$ divides $ B\left(U,V\right)$, which is a contradiction as we have already shown. Therefore, $u_{0} \ne 0$. Then, by considering the weights of $a_{\frac{p-1}{2}+pi}\left(s,t\right)$ and $a_{\frac{p^{2}-1}{2}}\left(s,t\right)$, we deduce that for some integers $0\le e\le 2$ and $0 \le e' \le 1$, the polynomials $\frac{a_{\frac{p-1}{2}+pi}\left(s,t\right)}{s^{e}t^{e'}}$ and $a_{\frac{p^{2}-1}{2}}\left(s,t\right)$ are weight-homogeneous of weights divisible by $6$. Hence, for some non-negative integer $e''$, the polynomials $\frac{a_{\frac{p-1}{2}+pi}\left(s,t\right)}{s^{e}t^{e'+2e''}}$ and $\frac{a_{\frac{p^{2}-1}{2}}\left(s,t\right)}{t^{\frac{p^{2}-1}{4}}}$ are weight-homogeneous in variable $s^{3}/t^{2}$ of weight $0$. We denote these by $g\left(s^{3}/t^{2}\right)$ and $h\left(s^{3}/t^{2}\right)$, respectively. Since $B\left(u_{0},u_{0}\right) = 0$, we have $a_{\frac{p-1}{2}}\left(u_{0},u_{0}\right) = 0$ from \eqref{eqn:leading_coeff}. By Proposition~\ref{prop:leading_coeff_and_other_coeff}, it follows that  $u_{0}^{e+e'+2e''} g\left(u_{0}\right) = a_{\frac{p-1}{2}+pi}\left(u_{0},u_{0}\right) = 0$ and $u_{0}^{\frac{p^{2}-1}{4}} h\left(u_{0}\right) = a_{\frac{p^{2}-1}{2}}\left(u_{0},u_{0}\right) \ne 0$. Since $u_{0} \ne 0$, we conclude that  $g\left(u_{0}\right) = 0$ and $h\left(u_{0}\right) \ne 0$. As $q_{0}\left(u\right)$ is irreducible, $q_{0}\left(u\right)$ divides $g\left(u\right)$ but not $h\left(u\right)$. Therefore, $Q\left(s,t\right)$ divides $a_{\frac{p-1}{2}+pi}\left(s,t\right)$ but not $a_{\frac{p^{2}-1}{2}}\left(s,t\right)$.
\end{proof}

We will need the following lemma  to show the irreducibility of $\theta_{s,t}\left(X\right)$ over $\Kst$.

\begin{lemma}\label{lem:Eisenstein-Gauss}
	Let $D$ be a Dedekind domain or a UFD, $P\subseteq D$ be a prime ideal, and $D_{P}$ be the localization of $D$ at $P$. If the polynomial $f\left(x\right) := \sum_{i=0}^{n} r_{i} x^{n-i}\in D_{P}[x]$ satisfy the conditions, $$
		r_{0} \notin \left(PD_{P}\right)^{2},\quad  r_{i} \in PD_{P} \text{ for all } 0 \le i\le n-1, \text{ and }  r_{n} \notin PD_{P},
	$$ then, $f$ is irreducible over the field of fraction of $D$.
\end{lemma}
\begin{proof}
	First, by  Eisenstein's criterion,  the polynomial $\sum_{i=0}^{n} r_{i} x^{i}$ is irreducible over $D_{P}$, and thus, so is $f$. Since $D_{P}$ is a UFD,  $f$ is irreducible over the field of fractions of $D_{P}$, by  Gauss's lemma(\cite[\S~9.Proposition~5]{Dummit-Foote}). As the field of fractions of $D$ is the same as that of $D_{P}$, the proof is completed.
\end{proof}

\begin{prop}\label{prop:theta_irr}
	The polynomial $\theta_{s,t}\left(X\right)$ obtained in Lemma~\ref{lem:poly1} is irreducible over $\Kst$.
\end{prop}
\begin{proof}
	Since $a_{\frac{p-1}{2}}$ is not constant (as  its weight is $\frac{p-1}{2}$), there exists a prime polynomial $Q \in K\left[s,t\right]$ that divides $a_{\frac{p-1}{2}}$. By Corollary~\ref{cor:leading_coeff_is_sep}, we have $Q \mid  a_{\frac{p-1}{2}+pi}$ for all $1\le i\le \frac{p-3}{2}$, $Q ~\|~ a_{\frac{p-1}{2}}$, and $Q \nmid a_{\frac{p^{2}-1}{2}}$. Thus, since $K\left[s,t\right]$ is a UFD, Lemma~\ref{lem:Eisenstein-Gauss} completes the proof.
\end{proof}

\begin{cor}\label{cor:first_extn}
    The field extension $\Kst\left(x\left(E_{s,t}\left[p\right]\right)\right) $ over $ \Kst$ has separable degree~$\frac{p-1}{2}$ and inseparable degree~$p$.
\end{cor}
\begin{proof}
	By Lemma~\ref{lem:poly1} and Proposition~\ref{prop:theta_irr}, if $u_{1}$ is a zero of $\overline{\psi_{p,s,t}}\left(x\right) = \theta_{s,t}\left(x^{p}\right)$, then, the extension $\Kst\left(u_{1}\right) / \Kst$ has separable degree $\frac{p-1}{2}$ and inseparable degree $p$. Moreover, $\Kst\left(x\left(E_{s,t}\left[p\right]\right)\right) = \Kst\left(u_{1}\right)$, since $E_{s,t}\left[p\right]$ is generated by a point whose $x$-coordinate is~$u_{1}$.
\end{proof}

Next, we consider the extensions $\Kst\left(x\left(E_{s,t}\left[p^{k+1}\right]\right)\right)/\Kst\left(x\left(E_{s,t}\left[p^{k}\right]\right)\right)$ for a $k\ge 1$. To investigate this, we first review normalized discrete valuations on a field~$F$. In this paper, a valuation refers to a normalized discrete valuation, which is a surjective group homomorphism $v : F^{\times} \twoheadrightarrow \bbz$ satisfying the property $$
	v\left(a+b\right) \ge \min\left\{v\left(a\right), v\left(b\right)\right\}
$$ for all $a,b \in F^{\times}$  with  $a+b\ne0$. Given this property, it makes sense to define $v\left(0\right) = \infty$. For $a,b \in F$, one can derive the following useful equality: \begin{equation}\label{eqn:valuation}
	v\left(a+b\right) = \min\left\{v\left(a\right),v\left(b\right)\right\}, \text{ if } v\left(a\right) \ne v\left(b\right).
\end{equation}

In our study of the extensions $\Kst\left(x\left(E_{s,t}\left[p^{k+1}\right]\right)\right)/\Kst\left(x\left(E_{s,t}\left[p^{k}\right]\right)\right)$, we focus on the valuations on the field of fractions of a Dedekind domain at a non-zero prime ideal.

In our setting, since $a_{\frac{p-1}{2}}$ is not constant, there exists a prime divisor $Q$ of $a_{\frac{p-1}{2}}$ in $K\left[s,t\right]$. Then,  we have that $Q \notin K\left[s\right]$ or $Q \notin K\left[t\right]$. We let \begin{equation}\label{R0}
	R_{0} := \begin{cases}
		 \left(K\left(s\right)\right)\left[t\right], \text{ if }Q \notin K\left[s\right],\\
	\left(K\left(t\right)\right)\left[s\right], \text{ if }Q \notin K\left[t\right] \text{.}
	\end{cases}
\end{equation} 
Note that both  $\left(K\left(s\right)\right)\left[t\right]$ and $\left(K\left(t\right)\right)\left[s\right]$ are Dedekind domains. By Gauss's lemma, $Q$ is a prime element in $R_{0}$. The valuations to be used in the proof of the following Proposition~\ref{prop:eta_irr} are those on extensions of the field of fractions of $R_{0}$ at primes lying above the prime ideal $QR_{0}$.

\begin{prop}\label{prop:eta_irr}
	For a prime divisor $Q$ of $a_{\frac{p-1}{2}}$ in $K\left[s,t\right]$, let $R_{0}$ be the Dedekind domain defined in \eqref{R0} and $F_{0}$ be the field of fractions of $R_{0}$. We choose a point $P_{1} \in E_{s,t}$ of order~$p$. Inductively, for each integer $k \ge 1$, we choose a point $P_{k+1} \in E_{s,t}$ of order $p^{k+1}$ such that $\left[p\right] P_{k+1} = P_{k}$. We let $u_{k} := x\left(P_{k}\right)$, $F_{k}:= F_{0}\left(u_{k}\right)$, and $R_{k}$ be the integral closure of $R_{0}$ in $F_{k}$. By the definition of $R_{0}$ in \eqref{R0},  $\fkQ_{0} := QR_{0}$ is a prime ideal. Inductively, we choose a prime ideal $\fkQ_{k} \subseteq R_{k}$ lying above $\fkQ_{k-1}$ for each $k\ge 1$. Then, we have that for each integer $k\geq 1$, the valuation of $u_{k}$ at $\fkQ_{k}$ is $-1$.
\end{prop}
\begin{proof}
	We prove by induction on $k$, following the notations from the proofs of Lemma~\ref{lem:poly1} and Lemma~\ref{lem:poly2}. For $k=1$, let $v$ denote the valuation on $F_{0}$ at $\fkQ_{0}$, $w$ the valuation on $F_{1}$ at $\fkQ_{1}$. There exists a positive integer $e$ such that $\left.w\right|_{F_{0}}= ev$. Since $u_{1}=x(P_1)$ is a zero of $\theta_{s,t}\left(x^{p}\right)$ by Lemma~\ref{lem:poly1}, we have \begin{equation}\label{eqn:org_theta}
		a_{\frac{p-1}{2}}u_{1}^{p\frac{p-1}{2}} + \left(\sum_{i=1}^{\frac{p-3}{2}}a_{\frac{p-1}{2}+pi}u_{1}^{p\frac{p-1}{2}-pi}\right) +  a_{\frac{p^{2}-1}{2}} = 0.
	\end{equation} 
	By Corollary~\ref{cor:leading_coeff_is_sep}, we see that $$
		v\left(a_{\frac{p-1}{2}}\right) = 1, \quad v\left(a_{\frac{p-1}{2}+pi}\right) \geq 1,  \text{ and } v\left(a_{\frac{p^{2}-1}{2}}\right) = 0.
	$$ Therefore, if $w\left(u_{1}\right)\ge 0$, then the only term with a non-positive $w$-valuation in the left hand side (LHS) of \eqref{eqn:org_theta} is $a_{\frac{p^{2}-1}{2}}$. By the property~\eqref{eqn:valuation}, the $w$-valuation of LHS of \eqref{eqn:org_theta} is $w\left(a_{\frac{p^{2}-1}{2}}\right) = 0$, which contradicts $w\left(0\right) = \infty$. Hence, $w\left(u_{1}\right) < 0$, equivalently, $w\left(u_{1}^{-1}\right) > 0$. To compute $w\left(u_{1}^{-1}\right)$, we divide both sides of \eqref{eqn:org_theta} by~$u_{1}^{p\frac{p-1}{2}}$:  \begin{equation}\label{eqn:org_theta_2}
		a_{\frac{p-1}{2}} + \left(\sum_{i=1}^{\frac{p-3}{2}}a_{\frac{p-1}{2}+pi}u_{1}^{-pi}\right) +  a_{\frac{p^{2}-1}{2}} u_{1}^{-p\frac{p-1}{2}} = 0.
	\end{equation} 
	For each $1 \le i \le \frac{p-3}{2}$,  Corollary~\ref{cor:leading_coeff_is_sep} shows that $$
		w\left(a_{\frac{p-1}{2}+pi}u_{1}^{-pi}\right) > w\left(a_{\frac{p-1}{2}+pi}\right) \ge w\left(a_{\frac{p-1}{2}}\right).
	$$
	If $w\left(a_{\frac{p-1}{2}}\right) \ne w\left(a_{\frac{p^{2}-1}{2}} u_{1}^{-p\frac{p-1}{2}}\right)$, then by the property~\eqref{eqn:valuation}, the $w$-valuation of LHS of \eqref{eqn:org_theta_2} would be $\min\left\{w\left(a_{\frac{p-1}{2}}\right), w\left(a_{\frac{p^{2}-1}{2}} u_{1}^{-p\frac{p-1}{2}}\right)\right\}$, which again contradicts  $w\left(0\right) = \infty$. Therefore, we must have $w\left(a_{\frac{p-1}{2}}\right) = w\left(a_{\frac{p^{2}-1}{2}} u_{1}^{-p\frac{p-1}{2}}\right)$. This implies   $p\frac{p-1}{2} \mid e$,
 as $$
	e = e v\left(a_{\frac{p-1}{2}}\right) = w\left(a_{\frac{p-1}{2}}\right) = w\left(a_{\frac{p^{2}-1}{2}} u_{1}^{-p\frac{p-1}{2}}\right) = p\frac{p-1}{2}w\left(u_{1}^{-1}\right).
	$$
 On the other hand, $e \le \left[F_{1}:F_{0}\right] = p\frac{p-1}{2}$ by Corollary~\ref{cor:first_extn}. Hence, $e = p\frac{p-1}{2} = \left[F_{1}:F_{0}\right]$ and $w\left(u_{1}^{-1}\right) = 1$.

	Now, suppose the statement is true for some $k\geq 1$ and we prove it for $k+1$. Since $P_{k} = \left[p\right] P_{k+1}$, by Lemma~\ref{lem:poly2} we have $$
		u_{k+1}^{p^{2}} + \left(\sum_{i=1}^{p-1}\left(b_{pi}+c_{pi-1}u_{k}\right) u_{k+1}^{p^{2}-pi}\right) + \left(b_{p^{2}}+c_{p^{2}-1}u_{k}\right) = 0,
	$$
	which can be rewritten as (noting that $u_k\neq 0$ from $v\left(u_{k}^{-1}\right) = 1$ by the induction hypothesis):
	\begin{equation}\label{eqn:sigma_int_form}
		u_{k}^{-1}u_{k+1}^{p^{2}} + \left(\sum_{i=1}^{p-1}\left(b_{pi}u_{k}^{-1}+c_{pi-1}\right) u_{k+1}^{p^{2}-pi}\right) + \left(b_{p^{2}}u_{k}^{-1}+c_{p^{2}-1}\right) = 0.
	\end{equation} 
	Let $v$ be the valuation on $F_{k}$ at $\fkQ_{k}$, $w$ be the valuation on $F_{k+1}$ at $\fkQ_{k+1}$. There exists a positive integer $e$ such that $\left.w\right|_{F_{k}}= ev$. For each $1 \le i\le p-1$, we have $w\left(b_{pi}\right), w\left(b_{p^{2}}\right) \ge 0$ since  $b_{pi}, b_{p^{2}} \in R_{0}$. By Lemma~\ref{lem:poly2} and  Corollary~\ref{cor:leading_coeff_is_sep}~(b), for each $1 \le i\le p-1$, we have $w\left(c_{pi-1}\right) > 0$. Since $c_{p^{2}-1} = -a_{\frac{p^{2}-1}{2}}^{2}$ by Lemma~\ref{lem:poly2},  Corollary~\ref{cor:leading_coeff_is_sep}~(c) implies  that $w\left(c_{p^{2}-1}\right) = 0$. Thus, we conclude \begin{equation} \label{eqn:ineq}
		w\left(b_{pi}\right), w\left(b_{p^{2}}\right)\ge 0, w\left(c_{pi-1}\right) \ge 1, \text{ and } w\left(c_{p^{2}-1}\right) = 0.
	\end{equation}
	So, since  $v\left(u_{k}^{-1}\right) = 1$ by the induction hypothesis, if $w\left(u_{k+1}\right) \ge 0$, the $w$-valuations of all terms except $c_{p^{2}-1}$ in LHS of \eqref{eqn:sigma_int_form} are positive. Since $w\left(c_{p^{2}-1}\right) = 0$, the property~\eqref{eqn:valuation} implies  the $w$-valuation of LHS of \eqref{eqn:sigma_int_form} is $0$, which contradicts $w\left(0\right) = \infty$. Therefore, $w\left(u_{k+1}^{-1}\right) > 0$. 
 
	To compute $w\left(u_{k+1}^{-1}\right)$, we divide both sides of \eqref{eqn:sigma_int_form} by $u_{k+1}^{p^{2}}$: \begin{equation}\label{eqn:sigma_inverse_form}
		u_{k}^{-1} + \left(\sum_{i=1}^{p-1}\left(b_{pi}u_{k}^{-1}+c_{pi-1}\right) u_{k+1}^{-pi}\right) + b_{p^{2}}u_{k}^{-1}u_{k+1}^{-p^{2}}+c_{p^{2}-1}u_{k+1}^{-p^{2}} = 0.
	\end{equation} 
	Since $w\left(u_{k+1}^{-1}\right) > 0$, and $v\left(u_{k}^{-1}\right) = 1$, referring to \eqref{eqn:ineq} we find that $$
		w\left(c_{pi-1}u_{k+1}^{-pi}\right) > w\left(c_{pi-1}\right) = ev\left(c_{pi-1}\right) \ge e = ev \left(u_{k}^{-1}\right) = w\left(u_{k}^{-1}\right),
	$$ 
and the $w$-valuations of  $b_{pi}u_{k}^{-1} u_{k+1}^{-pi}$, $b_{p^{2}}u_{k}^{-1}u_{k+1}^{-p^{2}}$, and $c_{pi-1} u_{k+1}^{-pi}$ exceed $w\left(u_{k}^{-1}\right)$. By~\eqref{eqn:valuation}, if $w\left(u_{k}^{-1}\right) \ne w\left(c_{p^{2}-1} u_{k+1}^{-p^{2}}\right)$, then the $w$-valuation of LHS of \eqref{eqn:sigma_inverse_form} is $\min \left\{ w\left(u_{k}^{-1}\right), w\left(c_{p^{2}-1} u_{k+1}^{-p^{2}}\right) \right\}$, which again contradicts $w\left(0\right) = \infty$. Therefore, $w\left(u_{k}^{-1}\right) = w\left(c_{p^{2}-1} u_{k+1}^{-p^{2}}\right)$. This implies  $p^{2} \mid e$, since $$
		e
		= ev\left(u_{k}^{-1}\right) 
		= w\left(u_{k}^{-1}\right)
		= w\left(c_{p^{2}-1}u_{k+1}^{-p^{2}}\right)
		= p^{2}w\left(u_{k+1}^{-1}\right).
	$$ On the other hand, $e \le \left[F_{k+1}:F_{k}\right] \le p^{2}$, since $F_{k+1}$ is generated by $u_{k+1}$ over $F_{k}$, and $u_{k+1}$ is a zero of a  polynomial $\eta_{s,t,u_{k}}\left(x^{p}\right)$ of degree $p^{2}$ defined over $F_{k}$ by Lemma~\ref{lem:poly2}. Hence, we have $e = \left[F_{k+1}:F_{k}\right] = p^{2}$ and $w\left(u_{k+1}^{-1}\right) = 1$.
\end{proof}

\begin{cor}\label{cor:eta_irr}
	For each integer $k\geq 1$, the normal extension $\Kst\left(x\left(E_{s,t}\left[p^{k+1}\right]\right)\right)$ over $\Kst\left(x\left(E_{s,t}\left[p^{k}\right]\right)\right)$ has separable degree $p$ and inseparable degree $p$.
\end{cor}
\begin{proof}
	Following the notations of Proposition~\ref{prop:eta_irr}, we note that the field $F_{k+1}$ is generated by $u_{k+1}$ over $F_{k}$ and $u_{k+1}$ is a zero of the polynomial \begin{equation}\label{eqn:sigma_inverse_poly}
		u_{k}^{-1}x^{p^{2}} + \left(\sum_{i=1}^{p-1}\left(b_{pi}u_{k}^{-1}+c_{pi-1}\right) x^{p^{2} - pi}\right) + \left(b_{p^{2}}u_{k}^{-1}+c_{p^{2}-1}\right)
	\end{equation} of degree $p^{2}$ in variable $x$ with coefficients in the localization $\left(R_{k}\right)_{\fkQ_{k}}$ since the valuation of $u_{k}^{-1}$ at $\fkQ_{k}$ is positive. Corollary~\ref{cor:leading_coeff_is_sep}, along with the proof of Proposition~\ref{prop:eta_irr}, and the definitions of the coefficients $c_{pi-1}$ and $c_{p^{2}-1}$ from Lemma~\ref{lem:poly2}, shows that  the conditions in Lemma~\ref{lem:Eisenstein-Gauss} for polynomial~\eqref{eqn:sigma_inverse_poly}  are satisfied, where $R_{k}$ is a Dedekind domain with a prime ideal $\fkQ_{k}$. Hence, we conclude that \eqref{eqn:sigma_inverse_poly} is irreducible over $F_{k}$ by Lemma~\ref{lem:Eisenstein-Gauss}, and so $\left[F_{k+1}:F_{k}\right] = p^{2}$. 

	If the inseparable degree of the extension $F_{k+1}/F_{k}$ is $p^{2}$, then $\eta_{s,t,u_{k}}\left(x^{p}\right)$ has only one zero. However, since $E_{s,t}$ is ordinary and $\eta_{s,t,u_{k}}\left(x^{p}\right)$ has distinct $p$ zeros,  the extension $F_{k+1}/F_{k}$ has separable degree~$p$ and inseparable degree~$p$. Moreover, $\Kst\left(x\left(E_{s,t}\left[p^{k}\right]\right)\right) = F_{k}$ and $\Kst\left(x\left(E_{s,t}\left[p^{k+1}\right]\right)\right) = F_{k+1}$, since $E_{s,t}\left[p^{k}\right]$ and $E_{s,t}\left[p^{k+1}\right]$ are generated by $P_{k}$ and $P_{k+1}$, respectively.
\end{proof}


Finally, we compute  the separable   and inseparable degrees of the  extension $\Kst\left(E_{s,t}\left[p^{n}\right]\right)$ over $\Kst$, and its automorphism group.

\begin{prop}\label{prop:modified_main}
	For each  integer $n\geq 1$, the normal  extension $\Kst\left(E_{s,t}\left[p^{n}\right]\right)$ over $\Kst$ has separable degree $p^{n-1}\left(p-1\right)$ and inseparable degree $p^{n}$.
\end{prop}
\begin{proof}
	Let $P_{n} = \left(x_{n},y_{n}\right)$ be a generator of $E_{s,t}\left[p^{n}\right] \cong \bbz/p^{n}\bbz$. Corollary~\ref{cor:first_extn} and  Corollary~\ref{cor:eta_irr} prove that $\Kst\left(x_{n}\right)$ over $\Kst$ has inseparable degree $p^{n}$ and separable degree $\frac{p-1}{2}p^{n-1}$. We show that $\left[\Kst\left(P_{n}\right):\Kst\left(x_{n}\right)\right] = 2$. Then, the extension $\Kst\left(E_{s,t}\left[p^{n}\right]\right)/\Kst$ has inseparable degree $p^{n}$ and separable extension $p^{n-1}\left(p-1\right)$.

	We suppose that $\Kst\left(P_{n}\right) = \Kst\left(x_{n}\right)$. It is equivalent to that  $$
		y_{n} \in \Kst\left(x_{n}\right).
	$$ Since the normal extension $\Kst\left(P_{n}\right)$ over $\Kst$ has separable degree $\frac{p-1}{2}p^{n-1}$, and $\left(\bbz/p^{n}\bbz\right)^{\times}$ has the unique subgroup $\left(\left(\bbz/p^{n}\bbz\right)^{\times}\right)^{2}$ of order $\frac{p-1}{2}p^{n-1}$, we have $\fkA_{E_{s,t}/\Kst}\left(n\right) \cong \left(\left(\bbz/p^{n}\bbz\right)^{\times}\right)^{2}$. Moreover, we conclude that  $p \equiv 3\pmod{4}$, since otherwise $-1 \in \left(\left(\bbz/p^{n}\bbz\right)^{\times}\right)^{2}$, which implies that there exists an automorphism $\sigma \in \fkA_{E_{s,t}/\Kst}\left(n\right)$ that fixes $x_{n}$ but not $y_{n}$, so $y_{n} \notin \Kst\left(x_{n}\right)$.

	Since the quadratic twist $E_{s^{3},s^{3}t}$ of $E_{s,t}$ by $s$ is also a specialization of $E_{s,t}$ at $(s^3, s^3t)$, we see that $\fkA_{E_{s^{3},s^{3}t}/\Kst}\left(n\right) \subseteq \left(\left(\bbz/p^{n}\bbz\right)^{\times}\right)^{2}$. This can be deduced as follows: if we consider $E_{s,t}$ as an elliptic curve defined over the field $\left(K\left(s_{*},t_{*}\right)\right)\left(s,t\right)$ for two variables $s_{*}$ and $t_{*}$, then $$
		\fkA_{E_{s^{3},s^{3}t}/\Kst}\left(n\right) = \fkA_{E_{s_{*}^{3},s_{*}^{3}t_{*}}/K\left(s_{*},t_{*}\right)}\left(n\right) \subseteq \fkA_{E_{s,t}/\left(K\left(s_{*},t_{*}\right)\right)\left(s,t\right)}\left(n\right) = \fkA_{E_{s,t}/\Kst}\left(n\right),
	$$ since $E_{s_{*}^{3},s_{*}^{3}t_{*}}$ is a specialization of $E_{s,t}$.
	Since $p\equiv 3\pmod{4}$, for any $r \in \left(\left(\bbz/p^{n}\bbz\right)^{\times}\right)^{2}$, $r \equiv \pm 1\pmod{p^{n}}$ implies $r \equiv 1\pmod{p^{n}}$. In other words, recalling that $\left(sx_{n},\sqrt{s^{3}}y_{n}\right)$ is a generator of $E_{s^{3}s^{3}t}\left[p^{n}\right]$, any automorphism in $\fkA_{E_{s^{3},s^{3}t}/\Kst}\left(n\right)$ fixing $sx_{n}$ must fix $\sqrt{s^{3}}y_{n}$, which implies that $\sqrt{s^{3}}y_{n} \in \Kst\left(sx_{n}\right) = \Kst\left(x_{n}\right)$ and $\sqrt{s} \in \Kst\left(x_{n}\right)$. Similarly, considering the quadratic twist of $E_{s,t}$ by $t$, we conclude $\sqrt{t} \in \Kst\left(x_{n}\right)$. Hence, there exists a biquadratic subextension  $K\left(\sqrt{s},\sqrt{t}\right)$ of $\Kst\left(x_{n}\right) = \Kst\left(P_{n}\right)$ over $K(s,t)$, which contradicts that the automorphism group $\fkA_{E_{s,t}/\Kst}\left(n\right)$ is isomorphic to a cyclic group $ \left(\left(\bbz/p^{n}\bbz\right)^{\times}\right)^{2}$.
\end{proof}

\

We are now ready to prove Theorem~\ref{thm:main}, from which we prove Corollary~\ref{cor:main1} and Corollary~\ref{cor:main2}.
\begin{proof}[Proof of  Theorem~\ref{thm:main}.]
	By Proposition~\ref{prop:modified_main}, the normal extension $\Kst\left(E_{s,t}\left[p^{n}\right]\right)$ over $K(s,t)$ has  separable degree $p^{n-1}\left(p-1\right)$, which is exactly the order of $\left(\bbz/p^{n}\bbz\right)^{\times}$, and  the automorphism group $\fkA_{E_{s,t}/\Kst}\left(n\right)$ is a subgroup of $\left(\bbz/p^{n}\bbz\right)^{\times}$. Therefore, $\fkA_{E_{s,t}/\Kst}\left(n\right) \cong \left(\bbz/p^{n}\bbz\right)^{\times}$.
\end{proof}

\

\begin{proof}[Proof of Corollary~\ref{cor:main1}.]
	First, for a variable $u$, we show that the automorphism group of the splitting field of $\overline{\psi_{p^{n},u,u}}$ over $K\left(u\right)$ is isomorphic to the automorphism group $\left(\left(\bbz/p^{n}\bbz\right)^{\times}\right)^{2}$ of the splitting field of $\overline{\psi_{p^{n},s,t}}\left(x\right)$ over $\Kst$. Since  $\overline{\psi_{p^{n},s,t}}\left(x\right)$ is weight-homogeneous of weight $\frac{p^{2n}-1}{2}$ by \eqref{homog}, we have $\overline{\psi_{p^{n},s^{3}/t^{2},s^{3}/t^{2}}}\left(x\right) = \left(\frac{s}{t}\right)^{\frac{p^{2n}-1}{2}} \overline{\psi_{p^{n},s,t}}\left(\frac{t}{s}x\right)$. Hence, the splitting fields of $\overline{\psi_{p^{n},s^{3}/t^{2},s^{3}/t^{2}}}$ and $\overline{\psi_{p^{n},s,t}}$ over $\Kst$ are the same. Moreover, $\Kst = K\left(s^{3}/t^{2},t\right)$ and two fields $K\left(s^{3}/t^{2}\right)$ and $K\left(t\right)$ are linearly disjoint  over $K$, and hence, the automorphism group of the splitting field $\overline{\psi_{p^{n},s^{3}/t^{2},s^{3}/t^{2}}}$ over $K\left(s^{3}/t^{2}\right)$ is isomorphic to$\left(\left(\bbz/p^{n}\bbz\right)^{\times}\right)^{2}$. Replacing $s^{3}/t^{2}$ by $u$, the automorphism group of the splitting field $\overline{\psi_{p^{n},u,u}}$ over $K\left(u\right)$ is isomorphic to $\left(\left(\bbz/p^{n}\bbz\right)^{\times}\right)^{2}$.

	By Hilbert's irreducibility theorem (\cite[Proposition~3.3.5]{Serre08} and \cite[Remark~2.4]{IK22}), there exists infinitely many distinct $u_{0} \in K$ such that the autmomorphism group of the splitting field $K\left(x\left(E_{u_{0},u_{0}}\left[p^{n}\right]\right)\right)$ of the polynomial $\overline{\psi_{p^{n},u_{0},u_{0}}}$ over $K$ is isomorphic to the autmomorphism group $\left(\left(\bbz/p^{n}\bbz\right)^{\times}\right)^{2}$ of the splitting field $K\left(x\left(E_{u,u}\left[p^{n}\right]\right)\right)$ of $\overline{\psi_{p^{n},u,u}}$ over $K\left(u\right)$. 
 
	For a generator $\left(x_{n},y_{n}\right)$ of $E_{u_{0},u_{0}}\left[p^{n}\right]$, if $y_{n} \notin K\left(x_{n}\right)$, then $\fkA_{E_{u_{0},u_{0}}/K}\left(n\right) \cong \left(\bbz/p^{n}\bbz\right)^{\times}$ by a similar argument as in the proof of Proposition~\ref{prop:modified_main}. 
 
	If $y_{n} \in K\left(x_{n}\right)$, we can find a $\delta \in K$ such that $\fkA_{E_{\delta^{2}u_{0},\delta^{3}u_{0}}/K}\left(n\right) \cong \left(\bbz/p^{n}\bbz\right)^{\times}$, by using Hilbert's irreducibility theorem. This theorem guarantees the existence of  infinitely many distinct $\delta \in \left(\bba^{1}\setminus \left\{0\right\}\right)\left(K\right)$ (up to square)  such that $\sqrt{\delta} \notin K$. Hence, for some such $\delta$, we have $\sqrt{\delta} \notin K\left(x_{n}\right)$, meaning that two extension fields $K\left(x_{n}\right)$ and $K\left(\sqrt{\delta}\right)$ are linearly disjoint over $K$. Consequently, the separable degree of the extension $K\left(x_{n}, \sqrt{\delta}\right)$ is equal to the order of $\left(\bbz/p^{n}\bbz\right)^{\times}$. Furthermore, $K\left(x_{n}, \sqrt{\delta}\right) = K\left(E_{\delta^{2}u_{0},\delta^{3}u_{0}}\left[p^{n}\right]\right)$, since $E_{\delta^{2}u_{0},\delta^{3}u_{0}}\left[p^{n}\right]$ is generated by $\left(\delta x_{n},\sqrt{\delta^{3}}y_{n}\right)$. In conclusion, $\fkA_{E_{\delta^{2}u_{0},\delta^{3}u_{0}}}\left(n\right) \cong \left(\bbz/p^{n}\bbz\right)^{\times}$. Since the $j$-invariant of $E_{\delta^{2}u_{0},\delta^{3}u_{0}}$ is $\frac{1728u_{0}}{u_{0}+27/4}$ which is determined uniquely by $u_{0}$, there are infinitely many distinct elliptic curves $E/K$ defined over $K$ such that $\fkA_{E/K}\left(n\right) \cong \left(\bbz/p^{n}\bbz\right)^{\times}$ with distinct $j$-invariants.
\end{proof}


\

\begin{proof}[Proof of Corollary~\ref{cor:main2}.]
	For an elliptic curve  $E/K$, if $E\left(K\right)\left[p\right]\neq 0$, then  $E\left(K\right)\left[p\right] \cong \bbz/p\bbz$, so $K\left(E\left[p\right]\right) = K$. Therefore, $\fkA_{E/K}\left(1\right)$ is trivial.

	By Corollary~\ref{cor:main1}, there exist infinitely many elliptic curves $E/K$ with distinct $j$-invariants such that $\fkA_{E/K}\left(1\right) \cong \left(\bbz/p\bbz\right)^{\times}$, which implies that $E\left(K\right)\left[p\right]=0$ from the above. Since there are only finitely many supersingular $j$-invariants by Lemma~\ref{lem:ssj}, this completes the proof.
\end{proof}

\end{document}